\documentclass{amsart}

\usepackage{amsmath,amssymb,amsthm}
\usepackage{csquotes}
\MakeOuterQuote{"}
\usepackage{url}
\usepackage{bbm}
\usepackage{graphicx,tikz}
\usetikzlibrary{shapes,arrows}
\usetikzlibrary{calc, positioning, shapes, through, intersections}
\newtheorem{theorem}{Theorem}
\newtheorem*{theorem*}{Theorem}
\newtheorem*{lemma*}{Lemma}

\newtheorem{lemma}[theorem]{Lemma}

\newtheorem*{question}{Question}

\theoremstyle{definition}

\theoremstyle{remark}

\makeatletter
\def\Let@{\def\\{\notag\math@cr}}
\makeatother
\newcommand{\was}{\rm Wass}

\begin{document}
\title[]{Spectral convergence of probability densities for forward problems in uncertainty quantification}

\author[]{Amir Sagiv}
\address{Department of Applied Physics and Applied Mathematics, Columbia University, New York, NY 10027, USA}
\email{as6011@columbia.edu}

\begin{abstract} 
The estimation of probability density functions (PDF) using approximate maps is a fundamental building block in computational probability. We consider forward problems in uncertainty quantification: the inputs or the parameters of a deterministic model are random with a known distribution. The scalar quantity of interest is a fixed measurable function of the parameters, and is therefore a random variable as a well. Often, the quantity of interest map is not explicitly known and difficult to compute, and so the computational problem is to design a good approximation (surrogate model) of the quantity of interest. For the goal of approximating the {\em moments} of the quantity of interest, there is a well developed body of research. One widely popular approach is generalized Polynomial Chaos (gPC) and its many variants, which approximate moments with spectral accuracy. However, it is not clear whether the quantity of interest can be approximated with spectral accuracy as well. This result does not follow directly from spectrally accurate moment estimation.  In this paper, we prove convergence rates for PDFs using collocation and Galerkin gPC methods with Legendre polynomials in all dimensions. In particular, exponential convergence of the densities is guaranteed for analytic quantities of interest. In one dimension, we provide more refined results with stronger convergence rates, as well as an alternative proof strategy based on optimal-transport techniques. 
\end{abstract}

\maketitle
\section{Introduction}

In the modeling of physical phenomena, taking uncertainties into consideration is often done by introducing random parameters into deterministic models. If the quantity of interest (model output) is a measurable function of the inputs, then it is a random as well. In this problem, known as the {\em forward problems} in uncertainty quantification (UQ), the computational task is to compute the induced probability measure of the quantity of interest given a prescribed probability measure on the inputs and parameters. The type of relevant information varies between applications, but in many cases the full probability density function (PDF) is of interest; See e.g., \cite{ablowitz2015interacting, chen2005uncertainty, colombo2018basins, 
estep2009nonparametric, zabaras2007sparse, le2010asynchronous, patwardhan2017loss,  best2017paper, piazzola2020uq, ullmann2014pod}. This paper will rigorously establish that there exist methods which approximate these PDFs with spectral accuracy.

One avenue to approximate the PDF in forward UQ problems is by nonparametric statistical approaches, e.g., a discretized cumulative distribution function, the histogram method, or kernel density estimators (KDE) \cite{tsybakov2009estimate, wasserman2006allnon}. 
	
To obtain better accuracy, one needs to utilize the underlying structure of forward UQ problems. That is, denote the input Borel probability space by $(\Omega, \varrho)$ and the Borel measurable model-output map by $f:\Omega \to \mathbb{R}$. The measure of interest $\mu=~f_{\sharp}\varrho$ is the {\em pushforward} of $\varrho$ by $f$, i.e., $\mu (A) = \varrho (f^{-1}(A))$ for every Borel set $A\subseteq~\mathbb{R}$. To take this structure into account, we turn to approximation-based methods (surrogate models). In this class of methods, $f$ is approximated by a simpler function $g$ for which the pushforward $g_{\sharp}\varrho$ is easier to compute \cite{ghanem2017handbook, xiu2010numerical}. For smooth quantities of interest with low- or moderate-dimensional domains, surrogate models have become the standard~\cite{ghanem2017handbook, sudret2000stochastic}, most prominently the polynomial-based methods known as generalized polynomial chaos~(gPC)~
\cite{ghanem2003stochastic, o2013polynomial, stefanou2009stochastic, xiu2010numerical} and their many variants, see e.g., \cite{lemaitre2004wavelet, tipireddy2014basis, wan2005adaptive}. The success of the gPC methods relies on their spectral $L^2(\Omega)$ convergence, when the original map is sufficiently smooth and specific domain and boundary conditions are met \cite{xiu2005collocation, xiu2002galerkin}. By spectral accuracy we mean the following: let $g_n$ be the gPC polynomial of order $n$ (see Sec.\ \ref{sec:pre}). Then $\|f-g_n\|_{L^2}$ decays polynomially in $n$ and the order of decay improves with the order of regularity of $f$. Crucially, the error decays {\em exponentially} in $n$ if $f$ is analytic. $L^2(\Omega)$ convergence implies convergence of the moments of $\nu _n :\,=(g_n)_{\sharp}\varrho$ to those of~$\mu$. Practically, this means that for analytic functions of interest $f$, the moments converge exponentially fast in the order of approximation.
\begin{question}
If $g_n$ is the gPC approximation of $f$, does $p_{\nu_n}$ converge to $p_{\mu}$ with spectral accuracy?
\end{question}
That a sequence of functions $g_n$ converges to $f$ in $L^2$ {\em does not guarantee} that the resulting PDFs $p_{\nu_n}$ converge to $p_{\mu}$ in any $L^q (\mathbb{R})$ norm, see counter-example in~\cite{sagiv2020uq}. Despite its practical relevance, the problem of density estimation using surrogate models has so far received little theoretical treatment, see one notable example in~\cite{butler2018convergence}. We previously showed that $L^2$ convergence of $g_n$ to $f$ guarantees convergence of $\nu_n $ to $\mu$ in the Wasserstein-$p$ metric, see details in \cite{sagiv2020wasserstein} and in Sec.~\ref{sec:wass}. Convergence in the Wasserstein metric, however, is in general weaker than convergence of the PDFs~\cite{gibbs2002choosing, le2012asymptotics}. To guarantee PDF approximation, Ditkowski, Fibich, and the author~\cite{sagiv2020uq} proved that more strict regularity conditions and $C^1$ convergence of $g_n$ to~$f$ guarantee convergence of $p_{\nu_n}$ to $p_{\mu}$ in $L^q(\mathbb{R})$ for all $1\leq q <\infty$ (see Theorem~\ref{thm:pf_local} below) and provided a spline-based method with provable convergence rates for the PDFs. Related, but not equivalent results, have also been obtained for certain monotonic triangular maps on $[-1,1]^d$ \cite{zech2020sparse}, and for the approximation of stochastic PDEs \cite{capodaglio2018approximation}. From a practical standpoint, it is desirable to have a surrogate model with spectrally convergent PDFs. One would like to prove that a well-known trade-off in function approximation (e.g., in $L^2$) holds for PDF approximation as well: on the one hand, local methods (like splines) are more robust to high derivatives and non-smoothness than spectral methods (like gPC). On the other hand, local methods are usually restricted to a fixed polynomial order of convergence, whereas spectral methods converge extremely fast for smooth functions. The question of spectral PDF convergence is therefore open and interesting.


In this work, we answer the above question affirmatively for the gPC expansion with respect to Legendre polynomials. To the best of our knowledge, these are the first results of this kind. Under high regularity conditions, the PDFs obtained by the gPC method converge spectrally (Theorem~\ref{thm:main_md}), and in particular converge exponentially fast for analytic quantities of interest. In the one-dimensional case we prove a stronger result, that $\|p_{\mu}-p_{\nu_n}\|_{L^1(\mathbb{R})} \lesssim \|f-g_n\|_{H^1}$ (Theorem~\ref{thm:main_1d}), and provide an alternative spectral convergence result (Theorem \ref{thm:l1_1dwass}), for which the proof relies on the weak Wasserstein-$p$ convergence combined with a recent result by Chae and Walker \cite{chae2020wass}. In all these results, the analysis relies on three main analytic components - the relation between $f$ and $p_{\mu}$ \eqref{eq:pdfder}, the approximation power of gPC in Sobolev spaces \eqref{eq:canuto}, and Sobolev-Morrey embeddings in compact domains~\eqref{eq:embedding}. In this work, we study pushforward measures induced by full-grid (or tensor-product grid) polynomial approximations. In high dimensions, the use of tensor-product grids becomes computationally expansive due to the so-called curse of dimensionality - the computational cost required to achieve a certain accuracy increases exponentially with the dimension. Several approaches have been developed to perform high-dimensional uncertainty propagation, such as sparse grids \cite{back2011sparse, bungartz2004sparse, xiu2010numerical, xiu2005collocation} and adaptive methods \cite{wan2005adaptive}. While these methods are beyond the scope is of this work, our analysis lays the numerical analysis foundation to understand the pushforward measures they induce, and thus their adequacy for density estimation.


\subsection{Paper structure} The remainder of the paper is organized as follows: Sec.\ \ref{sec:pre} provides the preliminaries on the gPC method, pushforward measures and a more detailed account of previous results. Secs.\ \ref{sec:main} and \ref{sec:pf} detail the paper's main theoretical results and their proofs, respectively. We then turn in Sec.\ \ref{sec:wass} to provide an alternative, optimal-transport based analysis of the measure pushforward problem. Finally, in Sec.\ \ref{sec:numerics} we present numerical experiments and outline open problems.

\section{Preliminaries}\label{sec:pre}

\subsection{Generalized Polynomial Chaos}\label{sec:gpc}
We present a brief review of the gPC method, see \cite{xiu2010numerical} for a thorough introduction. For clarity, we first review the gPC method for the one-dimensional case. The Legendre polynomials $\{ p_n (\alpha ) \}_{n=0} ^{\infty }$ are the set of orthogonal polynomials with respect to the Lebesgue measure on $\Omega = ~[-1,1]$, see \cite{szego1939orthogonal}
$$ {\rm Deg} ( p_n ) = n \, , \qquad  \langle p_n , p_m \rangle_{L^2(\Omega)} =  \delta _{n,m} \, . $$ The Legendre polynomials constitute an orthonormal basis of $L^2(\Omega, d\alpha)$, i.e., 
\begin{equation}\label{eq:pc_expansion}
f(\alpha ) = \sum\limits_{n=0}^{\infty } \hat{f} (n) p_n (\alpha ) \, , \qquad \hat{f}(n) := \langle p_n, f \rangle_{L^2(\Omega)} \, , \qquad f\in L^2(\Omega, d\alpha) \, . 
\end{equation}
The {\bf Galerkin-gPC} expansion of a function $f\in L^2(\Omega)$ of order $n$ is the projection of $f$ to its first $n+1$ modes, i.e., $g_n (\alpha) = \sum_{j=0}^{n} \hat{f}(n) p_n (\alpha)$. In the multidimensional case $\Omega = [-1,1]^d$, the expansion of order $n$ of $f$ is defined as the $L^2(\Omega)$ projection of $f$ into the space of polynomials of maximal degree $n$~\cite{canuto1982approximation, ghanem2003stochastic, xiu2002galerkin} 
\begin{equation}\label{eq:galerkin}
g_n =P_nf  = \sum\limits_{\|{\bf j}\|_{\infty} \leq n} \langle p_{\bf j}, f \rangle _{L^2(\Omega)} p_{\bf j} \, , \qquad n\in \mathbb{N} \, ,
\end{equation}
where $p_{\bf j}$ is the tensor-product Legendre polynomial of multi-index ${\bf j}=(j_1, \ldots, j_d )$ and $\|{\bf j} \|_{\infty} = \max_ i |j_i|$. If $f$ can only be evaluated at a discrete set of points, the integrals in expansion coefficients~$\hat{f}(n)$ cannot be exactly computed. The {\bf collocation-gPC} approach is to approximate these coefficients using the Gauss quadrature formula: Let $q:\Omega \to \mathbb{R}$ be an integrable function, then the Gauss-Legendre quadrature formula of order $n$ is
$\int_{-1}^1 q(\alpha)\, d\alpha \approx \sum_{k=1}^n g (\alpha _k ) w_k$,  where $\{ \alpha _k\}_{k=1}^{N}$ are the distinct and real roots of $p_N(\alpha)$, $w_k  :=\int_{\Omega } l_k  (\alpha ) \, d\mu (\alpha )$ are the weights, and $l_k  (\alpha)$ are the Lagrange interpolation polynomials with respect at the quadrature points~\cite{davis1967integration}. Taking $q=p_j(\alpha)f(\alpha)$ yields the gPC collocation method in one dimension~\cite{xiu2005collocation},
\begin{subequations}\label{eq:gpc_col}
\begin{equation}
g_N (\alpha) :\,= \sum\limits_{j=0}^{N-1} \hat{f} _N (j) p_j (\alpha) \, ,
\end{equation}
\begin{equation}\label{eq:coefs_with_quad}
\hat{f} (j)  \approx \hat{f}_N (j) :\,= \sum\limits_{k=1}^N f\left( \alpha _k  \right) p_j  \left( \alpha _k  \right) w_k ,\qquad j=0,1,\ldots, N-1 \, .
\end{equation}
\end{subequations} 
The generalization of \eqref{eq:gpc_col} to multiple dimensions is a direct result of integration by tensor-grid Gauss quadratures, see~\cite{canuto1982approximation, davis1967integration, xiu2005collocation} for details. We further note that the gPC-collocation polynomial \eqref{eq:gpc_col} is also the polynomial interpolant of degree $N-1$ in the Guass-Legendre points, i.e., $g_n(\alpha_{\bf k}) = f(\alpha_{\bf k})$ for all quadrature points \cite{constantine2012sparse}.

\subsection{The approximation power of Legendre polynomials}\label{sec:approx}
Recall that for $A= [-1,1]^d$ or $ \mathbb{R}^d$, and any pair of integers $k,p \geq 1$, the corresponding Sobolev space is defined as \cite{adams2003sobolev, evans_pde}
$$W^{k,p}(A) :\,= \left\{ u:A\to \mathbb{C} ~~ | ~~\max\limits_{\|{\bf j}\|_1 \leq k } \|D^{\bf j} u \|_{L^p(A)} < \infty  ~ \right\} \, ,$$
where ${\bf j}$ is a multi-index and $D^{\bf j} u = (\Pi_{k=1}^d \partial_{\alpha _k}^{j_k})u(\alpha) $. For $p=2$ we will adopt the conventional notation $W^{k,2}(A)=  H^k (A)$. Both the Galerkin~\eqref{eq:galerkin} and the collocation expansion \eqref{eq:gpc_col} converge spectrally in $L^2(\Omega)$, where the underlying measure is the Lebsegue measure \cite{gotlieb2007spectral, trefethen2013book, wang2012convergence, xiu2010numerical}. By this we mean that if $f \in H^{k}(\Omega)$ for some $k\geq 0$, then \cite{canuto1982approximation, gotlieb2007spectral, xiu2010numerical} $$\|f-g_n\|_{L^2(\Omega)} \leq cn^{-k}\|f\|_{H^k(\Omega)} \, ,$$ 
and if $f$ is analytic (in the multivariate sense), then $\|f-g_n\|_{L^2(\Omega)}$ decays {\em exponentially} in $n$ \cite{davis1975interpolation}, See \cite{zech2020sparse} and the references therein for details.
From a probability and UQ point of view, the convergence in $L^2(\Omega)$ guarantees that the moments of $g_n$ converge to the moments of $f$. Indeed, if $\varrho$ is the Lebesgue measure on $\Omega$, a simple application of the Cauchy-Schwartz inequality shows that $$|\mathbb{E}f - \mathbb{E}g_n| \leq \int_{\Omega} (f(\alpha) -g(\alpha)) \, d\alpha \lesssim  \|f-g_n\|_{L^2(\Omega)} \, ,$$
and therefore for e.g., a smooth quantity of interest $f$, the gPC method provides an accurate approximation of $\mathbb{E}f$ for a relatively low order $n$. The convergence of other moments follow similarly, see e.g., \cite{sagiv2020uq}.\footnote{For measures $\varrho$ which are not the Lebesgue measure, see \cite{ditkowski2019spectral}.} Both gPC expansions \eqref{eq:galerkin} and~\eqref{eq:gpc_col} also converge in higher-regularity Sobolev spaces, as given by the now classical result:
\begin{theorem*}[Canuto and Quarteroni \cite{canuto1982approximation}]
For any $1\leq \beta  \leq \sigma $, there exists a constant $C=C(\beta, \sigma)$ such that
\begin{subequations}\label{eq:canuto}
\begin{equation}
\|f - g_n\|_{H^{\beta}(\Omega)} \leq Cn^{-e(\beta ,\sigma)} \|f\|_{H^{\sigma}(\Omega)} \, ,
\end{equation}
where $g_n $ is given by either \eqref{eq:galerkin} or \eqref{eq:gpc_col}, and 
\begin{equation}\label{eq:e_canuto}
e(\beta,\sigma) = 
\sigma + \frac12 - 2\beta    \, .
\end{equation}
\end{subequations}
\end{theorem*}

 If $f$ is analytic (for $d=1$) on an ellipse in $\mathbb{C}$ with foci at $\pm 1$ and with both axis summing to~$r > 1$, then in one dimension $ \hat{f}(n) = O(n^{-1/2}r^{-n})$ \cite{wang2020fast, wang2012convergence}. By Parseval identity (with respect to Legendre polynomials),
$$\|f-g_n\|_{L^2(\Omega)}^2 = \sum\limits_{j=n+1}^{\infty} |\hat{f}(n)|^2 \lesssim \sum\limits_{j=n+1}^{\infty} n^{-1}r^{-2n} \, ,$$
and so the $L^2$ approximation error is almost exponential as well. Furthermore, by uniform convergence of $g_n$ to $f$ on $\Omega$, one can differentiate $g_n$ term-wise, yielding faster than polynomial (exponential for all practical purposes) convergence of $g_n$ to $f$ in all Sobolev spaces $H^k (\Omega)$ with $r\geq 0$.

\subsection{Pushforward measures and prior results}
The pushforward of a Borel measure $\varrho$ on $\Omega$ by a measurable $f:\Omega \to \mathbb{R}$ is a measure $\mu:\,=f_*\varrho$ defined by $\mu(A) = \varrho(f^{-1}(A))$ for any Borel subset $A\subseteq \mathbb{R}$. If a measure $\mu$ on $\mathbb{R}$ is absolutely continuous with respect to the Lebesgue measure, its probability density function (PDF) is its Radon-Nykodim derivative, i.e., $p_{\mu}\in L^1(\mathbb{R})$ which satisfies $\mu (A) = \int_{A} p_{\mu}(y) \, dy $ for any Borel set $A\subseteq \mathbb{R}$. Alternatively, if the cumulative distribution function (CDF) $F_{\mu}(y) = \mu(f^{-1}(-\infty, y))$ is differentiable, then $p_{\mu}(y) = dF_{\mu}(y)/dy$. In the one-dimensional case, if $f$ is piecewise $C^1$ and piecewise monotonic, and $\varrho$ is an absolutely continuous probability measure, then \cite{sagiv2020uq}
 \begin{equation}\label{eq:pdfder}
 p_{\mu}(y)= \sum\limits_{f(\alpha)=y} \frac{p_{\varrho}(\alpha)}{|f'(\alpha)|}\, . 
 \end{equation} This relation is the source of many of the difficulties and peculiarities in understanding PDFs	 of pushforward measures. For example, if $f(J)= c$ for some constant $c$ on an interval $J$ with $\varrho(J)>0$, then $\mu$ has a singular part at $c$ and therefore has no PDF. Even for a non-constant smooth monotonic function such as $f(\alpha)=\alpha^2$ and a simple $\varrho$ such as the uniform measure on $[0,1]$, then $p_{\mu}(y) \sim 1/\sqrt{y}$ is singular. Analogously for $d>1$, then  $p_{\mu}\sim \int_{f^{-1}(y)} p_{\varrho}|\nabla f|^{-1} d\sigma $, where $d\sigma$ is the $(d-1)$-dimensional surface elements.

The practical goal of surrogate models in the context of PDF approximation is to approximate $\mu$ by $\nu_n = (g_n)_{\sharp}\varrho$ with small error terms $\|p_{\mu}-p_{\nu _n}\|_{L^q(\mathbb{R})}$ for some $q\geq 1$, while maintaining $n$ small, as $n$ is a good proxy to the computational cost. To see that, first note that since sampling arbitrarily many times from the polynomial $g_n$ is computationally cheap, approximating $p_{\nu _n}$ to arbitrary precision given $g_n$ is relatively cheap as well, see the analysis in \cite{sagiv2020uq}. Therefore, it is constructing $g_n$ which is costly. In the collocation gPC \eqref{eq:gpc_col}, the larger the degree of approximation $g=g_n$ is, the more evaluations of $f$ are required. If for example $f$ models the response of a partial differential equation (PDE), each such evaluation amounts to a solution of the PDE, which is usually computationally expansive. In the Galerkin-type gPC \eqref{eq:galerkin}, one usually projects the original PDE with random parameters to several PDEs with deterministic parameters, the number of which also grows with the degree $n$, see \cite{xiu2010numerical} for details. Therefore, in either case guaranteeing good accuracy for a small degree $n$ should be the goal of our analysis.

 As noted in the introduction, $L^2(\Omega)$ convergence alone does not guarantee convergence of the pushforward PDFs. Indeed, one can construct a sequence such that $g_n \to f$ in $L^2(\Omega)$ but $\|p_{\nu_n}-p_{\mu}\|_{L^q(\mathbb{R})} >{\rm const}(q)$ for all $n,q\geq 1$ \cite{sagiv2020uq}. Previously however~\cite{sagiv2020wasserstein}, we proved that $L^2$ convergence is sufficient to establish convergence in the weaker Wasserstein metric, or more precisely, that $\was_2(f_{\sharp}\varrho,g_{\sharp}\varrho ) \leq \|f-g\|_{L^2(\Omega, \varrho)}$ (see relevant definitions in Sec.\ \ref{sec:wass}). Furthermore, we showed that uniform boundedness of $\|f-g_n\|_{\infty}$ combined with $L^2(\Omega)$ convergence guarantees Wasserstein-$p$ convergence for any $1\leq p <\infty$, see \eqref{eq:wass_inter} below. Since $\was_1(\mu,\nu)$ is equal to the $L^1(\mathbb{R})$ convergence of the CDFs \cite{salvemini1943sul, vallender1974calculation}, the convergence of the CDFs in $L^1$ is a corollary of our result. A result in the direction of PDF approximation by surrogate methods was obtain by Ditkowski, Fibich, and the author in \cite{sagiv2020uq}. If $g_n$ is taken to be the spline interpolant of $f$ of order $m$, on e.g., a uniform grid with a total of $N$ grid points, then $\|p_{\mu}-p_{\nu_n}\|_{L^q(\mathbb{R})} \lesssim N^{-m/d}$ for any $1\leq q < \infty$. The main characteristic of splines that is useful to establish this result is the {\em pointwise} $C^1(\Omega)$ convergence of $g_n$ to $f$, see Theorem \ref{thm:pf_local} below. We will use this tool tool later to prove Theorem~\ref{thm:main_md}.

\subsection{Problem formulation} Let $\Omega = \Omega (d) =[-1,1]^d$ be equipped with an absolutely continuous probability measure $\varrho$. Consider a smooth function of interest $f:\Omega  \to \mathbb{R}$ and let $g_n:\Omega\to \mathbb{R}$ be its generalized polynomial chaos (gPC) approximation - either its $L^2(\Omega)$ projection to the space of Legendre polynomials of order~$\leq n$ (Galerkin type) {\em or} its polynomial interpolant on the Gauss-Legendre quadrature points of order~$n$ (collocation type). Consider the pushforward measures $\mu :\,= f_* \varrho$ and $\nu_n :\,= (g_n)_* \varrho$ and denote their respective probability density functions (PDF) by $p_{\mu}$ and $p_{\nu_n}$. In what follows, we see under what conditions $p_{\nu_n}$ converges to~$p_{\mu}$ in $L^q(\mathbb{R})$ for $1\leq q < \infty$, and find the convergence rates.

\section{Main results}\label{sec:main}

The convergence of the pushed-forward PDFs is guaranteed by the following:

\begin{theorem}\label{thm:main_md}
Let $\Omega = [-1,1]^d$ for any $d\geq 1$, let $f\in  H^{\sigma }(\Omega)$ where 
\begin{equation}\label{eq:sigma_min}
\sigma \geq \sigma_{\min} (d)
= \left\{\begin{array}{ll}
5\frac12   +d \, , & d~{\rm even} \, , \\
4\frac12 +d \, , & d~{\rm odd} \, ,
\end{array} \right. 
\end{equation}
let $d\varrho(\alpha)= r(\alpha)d\alpha$ with $r\in C^1(\Omega)$, and assume that $|\nabla f|>\kappa_f >0$. Then for any $1\leq q <\infty$,
$$\|p_{\mu}-p_{\nu_n}\|_{L^q(\mathbb{R})} \lesssim \|f-g_n\|_{C^1(\Omega)} \lesssim n^{-\sigma + \sigma_{\min} -2} \|f\|_{H^{\sigma}(\Omega)} \, .$$
\end{theorem}
For the one-dimensional case, the result can be improved:
\begin{theorem}
\label{thm:main_1d}
Let $\Omega =[-1,1]$, let $f\in H^{\sigma}(\Omega)$ with $\sigma \geq 6$ and let $d\varrho(\alpha)= r(\alpha)d\alpha$ with $r\in C^1(\Omega)$. If $f'(\alpha_j)=0$ for finitely many points $\alpha_1, \ldots , \alpha _J \in \Omega $ and there exist $k_j\geq 2$ for each $1\leq j \leq J$ such that $|f^{(k_j+1)}(\alpha_j)| >0$, then 

\begin{equation}\label{eq:conv_sing}
\|p_{\mu}-p_{\nu_n}\|_{L^1(\mathbb{R})}\lesssim \|f-g_n\|_{H^1(\Omega)}^{\frac{1}{2k+1}}\lesssim \|f\|_{H^{\sigma}(\Omega)}^{\frac{1}{2k+1}} n^{-\frac{2\sigma-3}{2(2k+1)}}  \, ,
\end{equation}
where $k= \max_j k_j$. In particular, if $|f'(\alpha)|>\kappa_f >0$ for all $\alpha \in \Omega$, then
\begin{equation}\label{eq:conv1d}
\|p_{\mu}-p_{\nu_n}\|_{L^1(\mathbb{R})}\lesssim \|f-g_n\|_{H^1(\Omega)} \lesssim \|f\|_{H^{\sigma}(\Omega)} n^{\frac32 - \sigma } \,.
\end{equation}
Furthermore, if $f$ is analytic on an ellipse in $\mathbb{C}$ with foci at $\pm 1$ and with both radius summing to $r>1$, $\|p_{\mu}-p_{\nu}\|_{L^1(\mathbb{R})}$ converges faster than any polynomial in $n$. 
\end{theorem}

Theorem \ref{thm:main_1d} improves Theorem \ref{thm:main_md} in the case of $d=1$ in two ways - First, it allows for points with $f' = 0$. Second, Theorem \ref{thm:main_1d} improves the convergence rate by three orders, from $n^{4\frac12-\sigma}$ to $n^{\frac32 - \sigma}$. Generally, Theorem~\ref{thm:main_1d} also yields much smaller constants. These improvements are owed to our ability to link the PDF convergence directly to $H^1(\Omega)$ convergence of $g_n$ to $f$. Since the PDFs depend on $f'$, it is conjectured that convergence of $g_n$ to $f$ in $H^1(\Omega)$ is the weakest $H^k (\omega)$ convergence that guarantees convergence of the PDFs, thus possibly leading to a sharp rate in \eqref{eq:conv1d}. In Theorem \ref{thm:main_md}, we use $C^1(\Omega)$ convergence of $g_n$ to $f$, which by Sobolev embedding depends on the much stronger, and therefore much slower, $H^{2+\lfloor d/2\rfloor}(\Omega)$ convergence. The improvement of rates also implies the improvement of constants, as can be observed from the Canuto and Quarteroni result \eqref{eq:canuto}. Suppose one wishes to guarantee $\|p_{\mu}-p_{\nu_n}\|_1 \leq C n^{-9/2}$, and that $f' > \kappa_f >0$. Since Theorem \ref{thm:main_1d} uses $H^1$ convergence, the constant would increase with $\|f\|_{H^6}$, whereas in Theorem \ref{thm:main_md} it would depend on $\|f\|_{H^9}$. 

That the convergence constants depend on high-regularity Sobolev norm is emblematic of global methods in general, and spectral methods in particular: To approximate a function with high derivatives well, i.e., in the asymptotically guaranteed rate, one has to to use a gPC polynomial of a relatively high order. This high threshold resolution is often embodied in the constants of the upper bound. Here one draws a distinction between the spline-based surrogate proposed in \cite{sagiv2020uq} and the spectral methods: splines guarantee polynomial convergence with low-sensitivity to high-derivatives. Global polynomial methods, such as gPC, can provide exponential accuracy if $f$ is very smooth in comparison to the sampling resolution. Another way in which even Theorem \ref{thm:main_1d} depends on high-order Sobolev norms, is that it requires that $f$ is {\em at least} in $H^6$. This is a technical requirement that is needed to guarantee that $\|f\|_{C^2}$ does not grow with $n$, see Lemma \ref{lem:c2}. Simulations seem to suggest that this is not a sharp requirement, see Sec.\ \ref{sec:numerics}.


\section{Proofs of main results}\label{sec:pf}
Throughout this paper we need to establish the $C^1$ approximation of $f$ by $g_n$, and that $\|g_n\|_{C^2(\Omega)}$ is uniformly bounded in $n$.
\begin{lemma}\label{lem:c2}
Under the conditions of Theorem \ref{thm:main_md}, $\|g_n\|_{C^2(\Omega)}$ is uniformly bounded for all $n\in \mathbb{N}$, and 
\begin{equation}\label{eq:c1_conv}
\|f -g_n\|_{C^1(\Omega)} \lesssim n^{-\sigma +\sigma_{\min}+2} \|f\|_{H^{\sigma}(\Omega)} \, .
\end{equation} 
\end{lemma}
\begin{proof}
Recall the Sobolev-Morrey embedding theorem \cite{adams2003sobolev, evans_pde}: If $u\in H^s(\Omega)$, and $s>d/2$, then
\begin{equation}\label{eq:embedding}
\|u\|_{C^{s-\lfloor \frac{d}{2} \rfloor -1}(\Omega)} \lesssim \|u\|_{H^s(\Omega)} \, ,
\end{equation}
where $\lfloor x \rfloor$ is the lower integer value for any $x\geq 0$.\footnote{In this study, we will not use a stronger version of these embeddings for H{\"o}lder norms $C^{k,\alpha}$, but just the integer-power $C^k$ norms.} By \eqref{eq:embedding}, choosing $\beta= 3+\lfloor d/2 \rfloor$ yields
\begin{equation}\label{eq:embed_c2}\|f- g_n \|_{C^{2}(\Omega)}\lesssim \|f-g_n \|_{H^{\beta}(\Omega) } \, .
\end{equation}
Applying the Sobolev approximation result \eqref{eq:canuto} to \eqref{eq:embed_c2} yields
\begin{equation}\label{eq:c2_bdhsigma}
\|f-g_n \|_{C^2} \lesssim n^{-e(\beta, \sigma)} \| f \|_{H^{\sigma}} \, .
\end{equation}
To guarantee a uniform bound for all $n\in \mathbb{N}$, it is sufficient to choose $\sigma\geq \sigma_{\rm min}$ such that $e(\beta, \sigma_{\rm min}) =0$. Since $\beta >1$, then 
\begin{align*} 0&= e(\beta , \sigma_{\min}) \\
&= \sigma_{\min} + \frac12 -2\beta   \\
&= \sigma_{\min} + \frac12 - 2(3+\lfloor \frac{d}{2}\rfloor) \, . \\ &\Longrightarrow \sigma_{\min}
= \left\{\begin{array}{ll}
5\frac12 +d  \, , & d~{\rm even} \, , \\
4\frac12 +d \, , & d~{\rm odd} \, .
\end{array} \right. 
\end{align*}
For $\sigma \geq \sigma_{\min}$, then $e(\beta, \sigma)\geq 0$ in \eqref{eq:c2_bdhsigma}, and so $\|f-g_n \|_{C^2(\Omega)} \lesssim \|f\|_{H^{\sigma}(\Omega)}$. Hence, since $f\in C^2(\Omega_ d)$, then \begin{align*}
\|g_n \|_{C^2(\Omega)} &\lesssim  \|g_n-f\|_{C^2(\Omega)} + \|f\|_{C^2(\Omega)}\\
 &\lesssim \|f\|_{H^{\sigma}(\Omega)} + \|f\|_{C^2(\Omega)} \, , \qquad n\in \mathbb{N} \, .
\end{align*}
We proceed to prove the estimate \eqref{eq:c1_conv}. By Sobolev-Morrey inequality \eqref{eq:embedding} and \eqref{eq:canuto}, we have that 
\begin{align*}
\|f-P_n f\|_{C^1(\Omega)} & \lesssim \|f-P_n f \|_{H^{2+\lfloor \frac{d}{2} \rfloor}(\Omega)} \\
&\lesssim \|f\|_{H^{\sigma}} n^{e(2+\lfloor\frac{d}{2} \rfloor, \sigma)(\Omega)} \, ,
\end{align*}
where 
\begin{align*}
e(2+\lfloor\frac{d}{2} \rfloor, \sigma) &= \sigma +\frac{1}{2} - 4-2\lfloor \frac{d}{2} \rfloor \\
&=\left\{\begin{array}{ll}
\sigma - 3\frac12 -d  \, , & d~{\rm even} \, ,\\
\sigma - 2\frac12 - d  \, , &d~{\rm odd} \, ,
\end{array} \right.
\end{align*}
which is positive for all $\sigma > \sigma_{\min}$.
\end{proof}
\subsection{Proof of Theorem \ref{thm:main_md}}
Local $C^1$ convergence as established in Lemma \ref{lem:c2} implies convergence of PDFs by the following result:
\begin{theorem}[Corollary 5.5, \cite{sagiv2020uq}]\label{thm:pf_local}
Let  $(g_n)_{n=1}^{\infty} \subset C^1 (\Omega )$, and consider $f\in C^1(\Omega)$ such that $|\nabla f | > \kappa _f >0$. Then if
\begin{subequations}
\begin{equation}\label{eq:cond_c2bd}
\|g_n \|_{C^2(\Omega _d)} \leq K \, ,
\end{equation}
for some constant $K$ and for all $n\in \mathbb{N}$, and if
\begin{equation}\label{eq:cond_c1approx}
\|f-g_n \|_{C^1(\Omega)} \leq Kn^{-\tau} \, , \qquad \tau, K>0 \, ,
\end{equation}
\end{subequations} 
then $|p_{\mu}(y)-p_{\nu_n}(y)|\lesssim n^{-\tau}$ for all but $o(n^{-\tau})$ points, and therefore $$\|p_{\mu} - p_{\nu}\|_{L^q (\mathbb{R})} \lesssim n^{-\tau} \, ,$$
for all $1\leq q < \infty$. Furthermore, if $d=1$ and $g_n$ interpolates $f$ at the endpoints $f(\pm 1) = g_n(\pm 1)$, then the uniform estimate $\|f-g_n \|_{L^{\infty}} \lesssim n^{-\tau}$ holds.
\end{theorem}

Indeed, \eqref{eq:cond_c2bd} is guaranteed explicitly and \eqref{eq:c1_conv} guarantees \eqref{eq:cond_c1approx} with an explicit convergence rate~$\tau = \sigma -\sigma_{\min} - 2  >0$.

%
\subsection{Proof of Theorem \ref{thm:main_1d}}

For brevity, we omit the $n$ subscripts, denoting $g_n = g$ and $\nu = \nu_n$. First, we treat the case where $|f'|> \kappa_f >0$. Suppose $f$ is monotonic increasing, i.e., $f'>\kappa_f>0$. By the embedding theorem \eqref{eq:embedding}, that $f\in H^6(\Omega)$ implies that $f\in C^2(\Omega)$ and so 
 \begin{equation}\label{eq:pmu_mono}
 p_{\mu}(y)= \sum\limits_{f(\alpha)=y} \frac{p_{\varrho}(\alpha)}{|f'(\alpha)|} = \frac{r(f^{-1}(y))}{f'(f^{-1}(y))} \, ,
 \end{equation}
for every $y \in {\rm range}(f)$ \cite[Lemma 4.1]{sagiv2020uq}. Since $g$ is a polynomial, $g\in C^1(\Omega)$. Furthermore, by Lemma \ref{lem:c2} we have that $\|f-g\|_{C^1} \lesssim n^{-5/2}\|f\|_{H^4(\Omega)}$, and so for sufficiently large $n$, $g'(\alpha)> \kappa_f /2 > 0$ for all $\alpha \in \Omega$. Therefore $p_{\nu}(y) = r(g^{-1}(y))/g'(g^{-1}(y))$ for every $y \in {\rm range}(g)$. It might be that the ranges of $f$ and $g$, which are the supports of $p_{\mu}$ and $p_{\nu}$, respectively, do not overlap. Assume for simplicity that $f(-1)=g(-1)$ but $g(1)>f(1)$, the other cases can be treated similarly. Then 
\begin{equation}\label{eq:l1_err_break}
\|p_{\mu}-p_{\nu}\|_{L^1(\mathbb{R})} = \int\limits_{f(-1)}^{f(1)} |p_{\mu}(y)-p_{\nu}(y) | \, dy + \int\limits_{f(1)}^{g(1)} p_{\nu}(y) \, dy \, .
\end{equation}
We begin with the second integral - 
\begin{align*}
 \int\limits_{f(1)}^{g(1)} p_{\nu}(y) \, dy &=  \int\limits_{f(1)}^{g(1)} \frac{r(g^{-1}(y))}{g'(g^{-1}(y))} \, dy \\
 &\leq  \frac{2}{\kappa _f} \|r\|_{\infty} |g(1)-f(1)| \\
 &\leq  \frac{2}{\kappa _f} \|r\|_{\infty} \|g-f\|_{C^0(\Omega)} \\
  &\lesssim \frac{2}{\kappa _f} \|r\|_{\infty} \|g-f\|_{H^1(\Omega)} \, , \\  
\end{align*}
where the last inequality is due to the Sobolev-Morrey embedding \eqref{eq:embedding}. Therefore, we need only to consider the first integral in \eqref{eq:l1_err_break}, and we therefore assume without loss of generality that ${\rm range}(f)={\rm range}(g)$, and so
\begin{equation}\label{eq:l1_err_pre}
\|p_{\mu} - p_{\nu}\|_{L^1(\mathbb{R})} = \int\limits_{f(-1)}^{f(1)} \left|\frac{r(f^{-1}(y))}{f'(f^{-1}(y))} - \frac{r(g^{-1}(y))}{g'(g^{-1}(y))} \right| \, dy  \, .
\end{equation}
Denote $\alpha = f^{-1}(y)$ and $\alpha_*:\,= \alpha_* (\alpha) = g^{-1}(f(\alpha))$, then by change of variables
\begin{align*}
\|p_{\mu} - p_{\nu}\|_{L^1(\mathbb{R})} &= \int\limits_{-1}^{1} \left|\frac{r(\alpha)}{f'(\alpha)} - \frac{r(\alpha_*)}{g'(\alpha_*)} \right| f'(\alpha) \, d\alpha \\
&= \int\limits_{-1}^{1}\frac{|r(\alpha)g'(\alpha_*)-f'(\alpha)r(\alpha_*)|}{g'(\alpha^*)}   \, d\alpha  \, .
\end{align*}
Since $g'(\alpha)$ and $r(\alpha)$ are differentiable, then for any $\alpha\in \Omega$
\begin{align}\label{eq:tri_ineq_D}
|r(\alpha)g'(\alpha_*)-&f'(\alpha)r(\alpha_*)| 
\\ &\leq r(\alpha)|g'(\alpha_*)- g'(\alpha)| + r(\alpha) |g'(\alpha)-f'(\alpha)| + f'(\alpha)|r(\alpha)-r(\alpha_*)|\\
&\leq D|\alpha-\alpha_*| + r(\alpha)|f'(\alpha)-g'(\alpha)| \, , \qquad 
\end{align}
where $D:\,= \left[ \|r\|_{\infty} \|g''\|_{\infty} + \|r'\|_{\infty} \|f\|_{\infty}\right] $. In general, $D=D_n$ as $g=P_nf$ depends on $n$, and $\|g''\|_{\infty}$ might not be bounded. However, as in Lemma \ref{lem:c2}, $\|g''\|_{\infty}$ and therefore $D$ are uniformly bounded for all $n\geq 1$. Since $g' \geq \kappa_f/2$ for sufficiently large $n$, then substituting \eqref{eq:tri_ineq_D} in the integral yields
\begin{equation}\label{eq:I_II_def}
\|p_{\mu}-p_{\nu}\|_{L^1(\mathbb{R})} \lesssim  \frac{1}{\kappa_f}\underbrace{\int\limits_{-1}^1 r(\alpha)|f'(\alpha)-g'(\alpha)|\, d\alpha}_{:\,={\rm I}} + \underbrace{\frac{D}{\kappa_f}\int\limits_{-1}^{1}|\alpha-\alpha_*| \, d\alpha}_{:\,={\rm II}} \, .  
\end{equation}
Since $r(\alpha), (f'-g')\in L^2(\Omega)$, then by the Cauchy-Schwartz inequality ${\rm I}$ is bounded from above by
$$\left| \langle r, f'-g' \rangle_{L^2(\Omega)} \right|\leq \|r\|_{L^2(\Omega)} \cdot\|f'-g'\|_{L^2(\Omega)} \leq \|r\|_{L^2(\Omega)} \cdot \|f-g\|_{H^1(\Omega)} \, $$
To bound ${\rm II}$ in \eqref{eq:I_II_def} from above, we first note that by Lagrange's mean-value theorem, there exists $\beta$ between $\alpha$ and $\alpha_*$ such that $ g'(\beta)(\alpha-\alpha_*)= g(\alpha)-g(\alpha_*) =g(\alpha)-f(\alpha)$,
and therefore $|\alpha-\alpha_*|\leq |g(\alpha)-f(\alpha)|/\kappa_{\rm f}$. From here, the process of bounding ${\rm II}$ from above is the same as bounding ${\rm I}$, which yields ${\rm II}\leq D/\kappa_f \|f-g\|_{L^2(\Omega)}$. Therefore $\|p_{\mu}-p_{\nu}\|_{L^1(\mathbb{R})} \lesssim \|f-g\|_{H^1(\Omega)} $. Applying the relevant Sobolev approximation theorem \eqref{eq:canuto}, settles the case where $|f'|>\kappa_f >0$.

We now turn to the case where $g'$ and $f'$ vanish at finitely many points. As we show, it does not matter whether these zero points coincide or not, and we will therefore treat the case where $f'(-1)=0$, $g'(-1)>0$ and $f'(\alpha), g'(\alpha) >0$ for all $\alpha \in (-1,1]$. Assume without loss of generality that $f(-1)=g(-1)$.\footnote{We showed how to treat the case where the ranges of $f$ and $g$ do not overlap above.} Fix $\varepsilon>0$ and divide the integral for $\|p_{\mu}-p_{\nu}\|_1$ on the left hand side of \eqref{eq:l1_err_pre} into two domains - {\it (1)} isolating the singular point in the PDF, i.e., $y \in [g(-1), g(-1+\varepsilon)]$ and {\it (2)} the rest of the domain, $y\in [g(-1+\varepsilon), g(1)]$.
\begin{enumerate}
\item On the first domain $[g(-1), g(-1+\varepsilon)]$, we take a crude estimation 
\begin{subequations}\label{eq:eps_interval}
\begin{equation}
\int\limits_{g(-1)}^{g(-1+\varepsilon)} |p_{\mu}(y)-p_{\nu}(y) | \, dy \leq  \int\limits_{g(-1)}^{g(-1+\varepsilon)} p_{\mu}(y)+p_{\nu}(y)  \, dy \, .
\end{equation}
Taking, for example $p_{\mu}(y)$, by a change of variables $f(\alpha) = y$ we get 
\begin{equation}
\int\limits_{g(-1)}^{g(-1+\varepsilon)} p_{\mu}(y) = \int\limits_{g(-1)}^{g(-1+\varepsilon)} \frac{r(f^{-1}(y))}{f'(f^{-1}(y))} \, dy =\int\limits_{-1}^{-1+\varepsilon} r(\alpha) \, d\alpha \leq \|r\|_{\infty} \varepsilon \, .
\end{equation}
\end{subequations}
\item The integral on $[g(-1+\varepsilon), g(1)]$ reads the same as \eqref{eq:I_II_def}, where $\kappa_f$ is replaced by $\kappa_f^{\varepsilon} :\,= \min\limits_{x\in [-1+\varepsilon,1]} |f'|$, yielding 
\begin{equation}\label{eq:pl1_mono}
\int\limits_{g(-1+\varepsilon)}^{g(1)} |p_{\mu}(y)-p_{\nu}(y) | \, dy \leq \frac{\tilde{D}}{(\kappa_f^{\varepsilon})^2} \|f-g\|_{H^1} \, , 
\end{equation}
where $\hat{D} $ is $\varepsilon$-independent. How does $\kappa_f^{\varepsilon}$ depend on $\varepsilon$? Since in the worst case $f(-1)=f'(-1)=\cdots f^{(k)}(-1)=0$ but $f^{(k+1)}(-1)=0$, then by Taylor expansion, $\kappa_f^{\varepsilon} \approx \varepsilon^{k}$ for sufficiently small $\varepsilon$.
\end{enumerate}
By combining both upper bounds, we have that 
$$\|p_{\mu}-p_{\nu}\|_{L^1(\mathbb{R})} \lesssim  \varepsilon^{-2k} \|f-g\|_{H^1(\Omega)} + \|r\|_{\infty}\varepsilon \, .$$
This upper bound is minimized by equating its $\varepsilon$ derivative to zero, which yields
$\|r\|_{\infty} \sim 2k\|f-g\|\varepsilon^{-2k-1}$. Hence, for sufficiently small $\|f-g\|_{H^1}$ (i.e., for sufficiently large $n$), $\|p_{\mu}-p_{\nu}\|_{L^1(\mathbb{R})} \lesssim \|f-g\|_{H^1(\Omega)}^{1/(2k+1)} $.

Finally, we reduce the general case where $f'$ has finitely many nodal points $\alpha_1, \ldots ,\alpha_L$. In the integral $\int_{f(-1)}^{f(1)} |p_{\mu}(y)-p_{\nu}(y) | \, dy$, we take out the intervals $f((\alpha_j - \varepsilon, \alpha_j +\varepsilon))$ and treat them separately, as in \eqref{eq:eps_interval}. For $y$ value outside these intervals, we claim that $f'$ and $g'$ have the same sign, for sufficiently large $n$. We split the discussion into two cases. First, if $f'$ changes its sign at $\alpha_j$, then there are two points in the interval $(\alpha_j - \varepsilon,\alpha_j +\varepsilon)$ where $f'$ is maximized and minimized. Taking $n$ to be sufficiently large, then $g'$ has to be positive and negative at these points, respectively, due to the pointwise $C^1$ approximation \eqref{eq:c1_conv}. Therefore, $g'$ must change its sign in between. It might be that $g'$ changes its sign within this interval again, but outside of this interval, since $f'>\kappa_f ^{\varepsilon}>0$, then $g'$ has the same sign as $f'$ there for sufficiently large $n$. The second case is easier - if $f$ does not change its sign at $\alpha _j$, then outside the interval $f'>\kappa_f ^{\varepsilon}>0$, and so by \eqref{eq:c1_conv} $g' > \kappa_f^{\varepsilon}/2 > 0$.

\section{A transport-based convergence result for $d=1$}\label{sec:wass}
In this section, we present a different convergence result for the one-dimensional collocation gPC. This method of proof highlights the role that the "weaker" Wasserstein metric can play in understanding PDFs, and the potential such methods have for future works. The result only applied to the {\em Gauss-Lobatto} (GL), defined as the roots of $p_n '(\alpha)$, the derivative of the Legendre polynomial of order $n$, plus the endpoints $\{-1,1\}$. The polynomial interpolant at GL quadrature points admits the same Sobolev approximation theory as stated in~\eqref{eq:canuto} for the Gauss-Legendre points, see \cite{canuto1982approximation} and \cite[Chapter 10]{quarteroni2000numerical}. As we will see, the following result is restricted to the GL interpolant since the condition $f(\pm 1)=g_n(\pm 1)$ guarantees that ${\rm range}(f)={\rm range}(g_n)$.

\begin{theorem}\label{thm:l1_1dwass}
Let $\Omega =[-1,1]$. For any integer $m \geq 1$, and any function $f\in  W^{\sigma ,2}(\Omega)$ with $\sigma \geq 2m+4$ and $|f'|\geq \kappa _f >0$, let $g_n$ be its polynomial interpolant at the GL quadrature points of order $n$, and suppose the $d\varrho(\alpha) = r(\alpha) d\alpha $ where $r\in 	W^{m,1}(\Omega)$. Then 
\begin{equation}\label{eq:tv_wk1n}
\|p_{\mu}-p_{\nu_n}\|_{L^1(\mathbb{R})}\lesssim \|p_{\mu}\|_{W^{m,1}(A)}^{\frac{1}{m+1}}\|f\|_{W^{\sigma,2}(\Omega )}^{\frac{1}{m+1}} n^{-\frac{m}{m+1} \left(\sigma -\frac56\right)} \, ,
\end{equation}
where $A= {\rm image}(f)$.
\end{theorem}

For analytic functions, both Theorems \ref{thm:main_1d} and \ref{thm:l1_1dwass} guarantee faster than polynomial convergence. For functions $f\in H^{\sigma}\setminus H^{\sigma -1}$, however, Theorem \ref{thm:main_1d} guarantees slightly better convergence rates. Theorem \ref{thm:l1_1dwass} does improve on Theorem \ref{thm:main_1d} in that it relaxes the demand $r\in C^1$ to $r\in W^{m,1}$. But as noted, the importance of Theorem \ref{thm:l1_1dwass} is the method of its proof, which relies on the Wasserstein distance. 

Given two probability measures $\omega _1$ and $\omega _2$ on $\mathbb{R}$, the Wasserstein-$1$ distance is defined as\footnote{To avoid confusion - $W^{k,p}$ denotes the Sobolev spaces of functions with $k$ derivatives which are $p$ integrable, and $\was_p$ denotes the Wasserstein-$p$ distance (instead of the standard $W_p$). }
\begin{subequations}\label{eq:wass}
\begin{equation}
\was_1 (\omega_1, \omega_2 ) :\,=  \inf \limits_{\gamma \in \Gamma} \int\limits_{\mathbb{R}^2} |x-y| \, d\gamma (x,y)   \, , 
\end{equation}
where $\Gamma$ is the set of all measures $\gamma$ on $\mathbb{R}^2$ for which $\omega_1 $ and $\omega_2$ are marginals, i.e., for any Borel $B\subseteq \mathbb{R}$,
\begin{equation}\label{eq:marginals}
\omega_1 (B) = \int\limits_{\mathbb{R}\times B} \gamma(x,y) \, dy \, , \qquad \omega _2(B) = \int\limits_{B\times \mathbb{R}} \gamma(x,y) \, dx \, .
\end{equation}
\end{subequations}
Since $\omega_1(\mathbb{R})=\omega_2(\mathbb{R})=1$, a minimizer of \eqref{eq:wass} exists, and so $\was_1(\omega_1, \omega_2)$ is finite, and it is a metric \cite{santa2015optimal, villani2003topics}.
Intuitively, the Wasserstein distance is often referred to as the earth-mover's distance; $\was_1(\omega_1, \omega_2)$ computes the minimal work (distance times force) by which one can transfer a mound of earth in the mold of $\omega_1$ to a one that is in the mold of $\omega _2$.

How does the Wasserstein distance relate to the problem of densities? In general, $\was_1(\omega_1, \omega_2)\lesssim \|p_{\omega_1}-p_{\omega_2}\|_{L^1(\mathbb{R})}$, but not the other way around \cite{gibbs2002choosing}. This is why, in general, the Wasserstein distance induces a weaker topology on the space of probability measures than that induced by the $L^1$ distance between the PDFs. Moreover, the Wasserstein distance is even well-defined for singular measures, which do not have densities at all. A recent result due to Chae and Walker, however, shows that if the densities are sufficiently regular, the Wasserstein-$1$ metric bounds from above the $L^1$ distance between the densities.

\begin{theorem*}[Chae and Walker \cite{chae2020wass}]
Let $\omega_1$ and $\omega_2$ be two Borel measures on $A\subseteq \mathbb{R}$ with PDFs $p_{\omega_1}, p_{\omega_2} \in W^{m,1}(A)$ for some $m\geq 1$. Then 
\begin{equation}\label{eq:cw20}
\|p_{\omega_1}-p_{\omega_2}\|_{L^1(A)} \lesssim \left( \|p_{\omega_1}\|_{W^{m,1}(A)} +\|p_{\omega_2}\|_{W^{m,1}(A)}\right)^{\frac{1}{m+1}} \was_1^{\frac{m}{m+1}} (\omega_1, \omega_2) \, .
\end{equation}
\end{theorem*}
As we will show in Lemma \ref{lem:p_regularity}, one can verify under what conditions $p_{\mu}$ and $p_{\nu _n}$ are sufficiently regular. Therefore, to prove $L^1(\mathbb{R})$ convergence of $p_{\nu_n}$ to $p_{\mu}$, it is sufficient to prove the convergence of $\was_1(\mu, \nu _n)$. The weak convergence in the Wasserstein metric has recently been established under much more general conditions by the author:

\begin{theorem*}[\cite{sagiv2020wasserstein}]
For any compact Borel set $\Omega \subset \mathbb{R}^d$, and for every $1\leq p,q < \infty$,  
\begin{equation}\label{eq:wass_inter}
\was_p(\mu,\nu_n) \lesssim \|f-g_n\|_{L^{q}(\Omega )}^{\frac{q}{q+p}}\|f-g_n\|_{L^{\infty}(\Omega)}^{\frac{p}{q+p}} \, ,
\end{equation}
where the implicit constant depends only on $\Omega$, $p$, and $q$.
\end{theorem*}
 
\begin{proof}[Proof of Theorem \ref{thm:l1_1dwass}]
To apply~\eqref{eq:cw20} to $\mu$ and $\nu_n$, we need to show that their densities are sufficiently regular.
\begin{lemma}\label{lem:p_regularity}
Let $m\geq 1$ and let $f \in W^{2m+4,2}(\Omega)$ with $|f'(\alpha)|\geq \kappa_f > 0$ for all $\alpha \in I$, and denote $A={\rm range}(f)\subset \mathbb{R}$. Then $p_{\mu}\in W^{m,1}(A)$, $p_{\nu_n}\in W^{m,1}(A)$ for sufficiently large $n$, and $$\|p_{\nu_n}\|_{W^{m,1}(A)} \lesssim \|p_{\mu}\|_{W^{m,1}(A)} \, , $$
where the implicit constant does not depend on $n$.
\end{lemma}\label{lem:p_w11}
\begin{proof}[proof of lemma]
 We begin, for simplicity, with $m=1$. By \eqref{eq:pdfder}, if $f$ is monotonic, formally,
\begin{equation}\label{eq:pdf_der}
\frac{d}{dy}p_{\mu}(y) = -r(f^{-1}(y))\frac{f^{
''}(f^{-1}(y))}{|f'(f^{-1}(y))|^3} + \frac{r'(f^{-1}(y))}{|f'(f^{-1}(y))|^2} \, .
\end{equation}
By Sobolev-Morrey embedding \eqref{eq:embedding}, $f\in W^{6,2}(\Omega)\subseteq  C^2(\Omega)$. Combined with the fact that $|f'|>\kappa_f >0$, then $\frac{d}{dy}p_{\mu}(y)\in W^{1,1}(\mathbb{R})$. For $g_n$, due to Lemma \ref{lem:c2}, $|g_n '(\alpha)|>\kappa_f /2 >0$ on $\Omega$ for all sufficiently large $n$, and $\|g_n''\|_{\infty}$ is also uniformly bounded for all $n\geq 1$ ($g_n$ is a polynomial, and therefore smooth). By the analog of~\eqref{eq:pdf_der} for $\nu_n$ and $g_n$, we have that $p_{\nu_n} \in W^{1,1}(\mathbb{R})$. Since $|g_n''|$ is uniformly bounded in terms of $f$ and its derivatives, $\|p_{\nu_n}\|_{W^{1,1}(\mathbb{R})}\lesssim \|p_{\mu}\|_{W^{1,1}(\mathbb{R})}$ for all sufficiently large $n$. Here it is key that our domain is one-dimensional, that $g_n$ interpolates $f$ at the endpoints $\alpha = \pm 1$, and that both functions are monotonic. These three facts ensure that $$A={\rm supp}(\mu) = {\rm image}(f) = {\rm image}(g_n) = {\rm supp}(\nu_n) \, .$$
Suppose otherwise that, without loss of generality, $\max f(\alpha) \geq \max g_n (\alpha) = y_*$. In this case $p_{\nu _n}$ would generically have a step-like discontinuity at $y_*\in A$, and therefore $p_{\nu_n}\not\in W^{1,1}(A)$.

For a general $m\geq 1$, by direct differentiation one has that $\frac{d^m}{dy^m} p_{\mu} (y)$ is a sum of rational functions where the numerators depend on $f', \ldots ,f^{(m+1)}$ and in $r, \ldots, r^{(m)}$, and the denominators are monomials in $f'$. One then generalizes Lemma \ref{lem:c2} to show that adequately-high derivatives of $g_n$ are uniformly bounded in $n$, which concludes the Lemma.
\end{proof}

Lemma \ref{lem:p_w11} implies that \eqref{eq:cw20} is applicable in our settings, and that $(\|p_{\mu}\|_{W^{m,1}(A)} +\|p_{\nu_n}\|_{W^{m,1}(A)})^{1/m+1} \lesssim \|p_{\mu}\|_{W^{m,1}(A)}^{1/m+1}$, where the upper bound depends only on $f$ and its derivatives. Therefore
\begin{equation}\label{eq:cw20_applied}
\|p_{\mu}-p_{\nu_n}\|_{L^1(\mathbb{R})} = \|p_{\mu}-p_{\nu_n}\|_{L^1(A)} \lesssim \|p_{\mu}\|_{W^{m,1}(A)}^{\frac{1}{m+1}} \was_1^{\frac{1}{m+1}} (\mu, \nu_n) \, .
\end{equation}
Since $\Omega$ is compact, \eqref{eq:wass_inter} can be applied to $\was_1(\mu,\nu_n)$ such that \eqref{eq:cw20_applied} yields
\begin{equation}\label{eq:l1_bdinfty2}
\|p_{\mu}-p_{\nu_n}\|_{L^1(\mathbb{R})} \lesssim \|p_{\mu}\|_{W^{m,1}(A)}^{\frac{1}{m+1}} \|f-g_n\|_{L^{2}(\Omega )}^{\frac{2}{3(m+1)}}\|f-g_n\|_{L^{\infty}(\Omega)}^{\frac{1}{3(m+1)}} \, .
\end{equation}
By the spectral $L^2$ convergence, $\|f-g_n\|_{L^2} \lesssim W^{\sigma,2}n^{-\sigma}$, see \eqref{eq:canuto}. To bound the $L^{\infty}$ error, we use \eqref{eq:embedding} in conjuction with \eqref{eq:canuto} again, yielding $\|f-g_n\|_{C^0} \lesssim  \|f\|_{W^{\sigma, 2}}n^{-\sigma +2\frac12}$. Applying both of these upper bounds to \eqref{eq:l1_bdinfty2} than yields
\begin{align*}
\|p_{\mu}-p_{\nu_n}\|_{L^1(\mathbb{R})} 
&\lesssim  \|p_{\mu}\|_{W^{m,1}(A)}^{\frac{1}{m+1}}\|f\|_{W^{\sigma,2}(\Omega)}^{\frac{1}{m+1}} n^{-\frac{m}{m+1}(\sigma - \frac56)} \, ,
\end{align*}

\end{proof}

The following heuristic argument suggests that the condition $|f'|>\kappa_f >0$ is necessary for the proof of Theorem \ref{thm:l1_1dwass}. Let $r\equiv 1/2$ be the uniform density, and suppose without loss of generality that $f$ is monotonic increasing, that $f(0)=f'(0)=0$, and that by Taylor expansion $f(\delta) = c\delta^k + o (\delta ^k)$ for some integer $k\geq 2$ and $|c|>0$ as $\delta \to 0$. Then $f'(\delta) = kc\delta^{k-1} + o(\delta ^{k-1})$ and by direct substitution into \eqref{eq:pdf_der} $$\frac{d}{dy} p_{\mu}'(y) \sim y^{-2 + 1/k} \, ,\qquad y\to 0 \, ,$$ which is not integrable in any neighborhood of $y=0$. Hence, $p_{\mu}\not\in W^{k,1}(A)$ for any $k\geq 1$, and we cannot use \eqref{eq:cw20}.

\section{Numerical experiments and open questions}\label{sec:numerics}

\begin{figure}[h]
\centering
{\includegraphics[scale=.3]{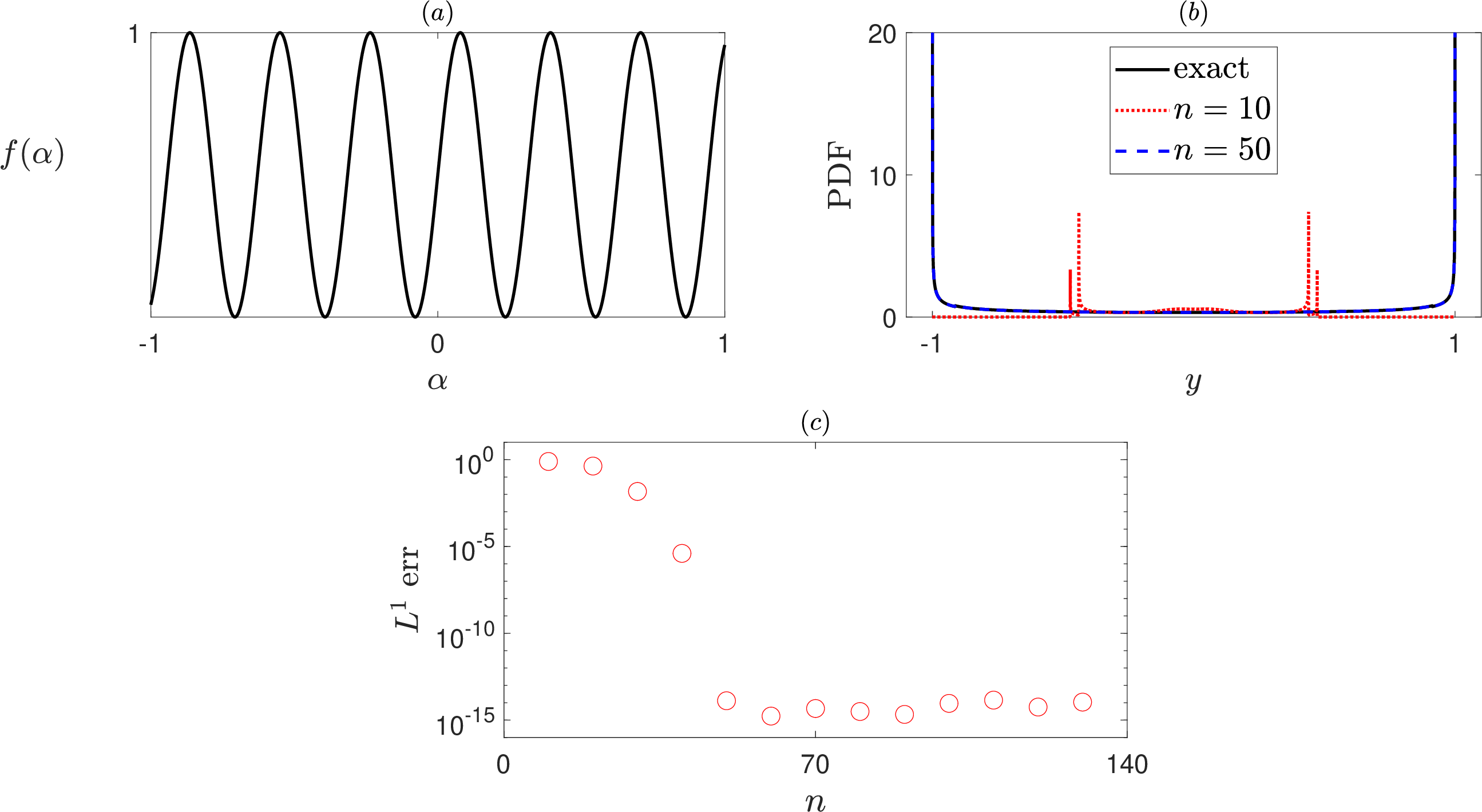}}
\caption{PDF approximation for $\mu = f_{\sharp}\varrho$ where $f(\alpha) = \sin (20\alpha)$ and $\varrho$ is the uniform probability measure on $[-1,1]$. (a)~$f(\alpha)$. (b) $p_{\mu}$ (solid, black), $p_{\nu_{10}}$ (dots, red) and $p_{\nu_{50}}$ (dashes, blue), where $\nu_n = (g_n)_{\sharp}\varrho$ where $g_n$ is the gPC collocation approximation of order $n$. $p_{\mu}$ and $p_{\nu_{50}}$ are nearly indistinguishable. (c)~$\|p_{\mu}-p_{\nu_n}\|_1$ as a function of~$n$. }
\label{fig:sin}
\end{figure}

We highlight some aspects of density-approximation using the gPC collocation method \eqref{eq:gpc_col} by numerical experiments on $\Omega = [-1,1]$ where $d\varrho (\alpha) = 1/2 \, d\alpha$ is the uniform probability measure on $\Omega$. We first consider $f(\alpha)=\sin (20\alpha)$, see Fig.~\ref{fig:sin}(a). The approximation in $L^2(\Omega)$ and $H^1(\Omega)$ of $f$ by polynomials is quite standard: a small number of collocation points (conversely, a low-order polynomial) does not suffice to resolve the oscillations of $f$. Once $n$ is sufficiently large, since $f$ is analytic, we expect $g_n$ to converge to $f$ in $H^s(\Omega)$ exponentially fast, for every $s\geq 0$. For the PDF, we see that indeed $p_{\nu_{10}}$ approximates $p_{\mu}$ quite poorly, whereas $p_{\nu_{50}}$ is nearly indistinguishable from $p_{\mu}$; see Fig.\ \ref{fig:sin}(b). Quantitatively, the $L^1(\mathbb{R})$ error between the PDFs follows the expected pattern - no convergence for $n\leq 30$, but then a sharp, exponential decay until machine-precision is reached; see Fig.\ \ref{fig:sin}(c). Another interesting facet of this example is that since $f'=0$ at several points, $p_{\mu}$ is singular at $\pm 1$, see \eqref{eq:pdfder}. Theorem \ref{thm:main_1d} therefore implies that $\|p_{\mu}-p_{\nu_n}\|_{L^1(\mathbb{R})} \lesssim \|f-g_n\|_{H^1(\Omega)}^{1/3}$ since $f''\neq 0$ at the maximas and minimas, see \eqref{eq:conv_sing}. However, since $\|f-g_n\|_{H^1(\Omega)}$ decays exponentially with $n$, the effect of this loss of accuracy is hardly noticeable.

\begin{figure}[h]
\centering
{\includegraphics[scale=.35]{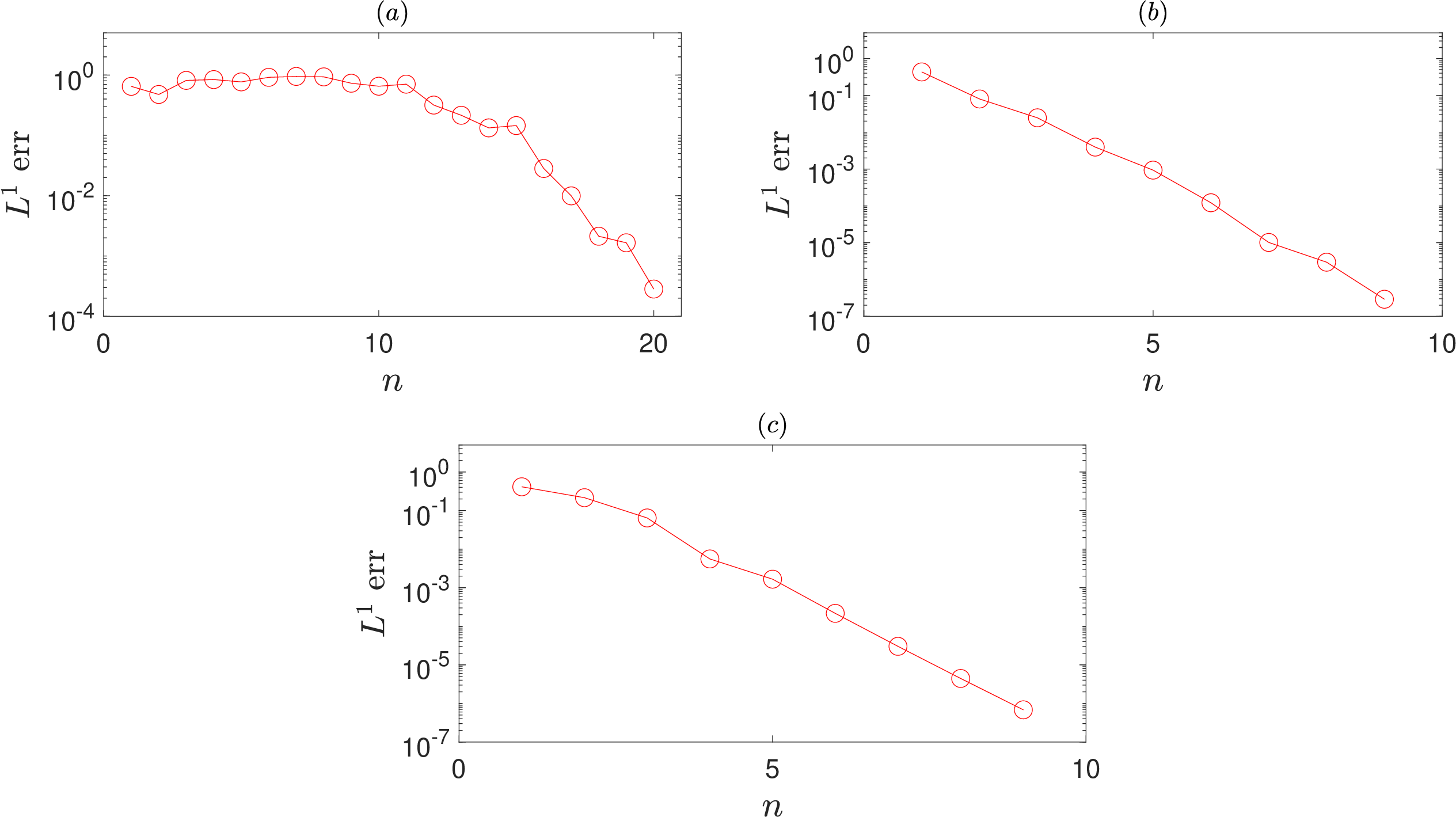}}
\caption{$\|p_{\mu}-p_{\nu_n}\|_1$ as a function of $n$ in higher dimensions. (a) $f(x,y)=\sin (10(x+y))$. (b) $f(x,y)=\sin (x+y)$. (c) $f(x,y,z)=~\sin (x+y+z)$.}
\label{fig:multid}
\end{figure}

We repeat the same experiment in higher dimension, using the gPC interpolant on the Gauss-Legendre tensor-product quadrature points (using the Sparse Grid Matlab Kit \cite{back2011sparse}). First, for $d=2$, we test $f(x,y)=\sin (10(x+y))$, for which we get similar behavior as we got for $d=1$; compare Fig.\ \ref{fig:sin}(c) with Fig.\ \ref{fig:multid}(a). Next, to observe ``simpler'' exponential convergence, we test in $d=2,3$ the functions $f(x,y)=\sin (x+y)$ and $f(x,y,z)=\sin (x+y+z)$, in Figs.\ \ref{fig:multid}(b) and \ref{fig:multid}(c), respectively. Note that in all three figures the error convergence is plotted against $n$, the maximal polynomial degree, see e.g., \eqref{eq:galerkin}. The computational cost, however, as represented by the number of grid point, hence the number of evaluations of $f$, scales is $(n+1)^d$.

\begin{figure}[h]
\centering
{\includegraphics[scale=.3]{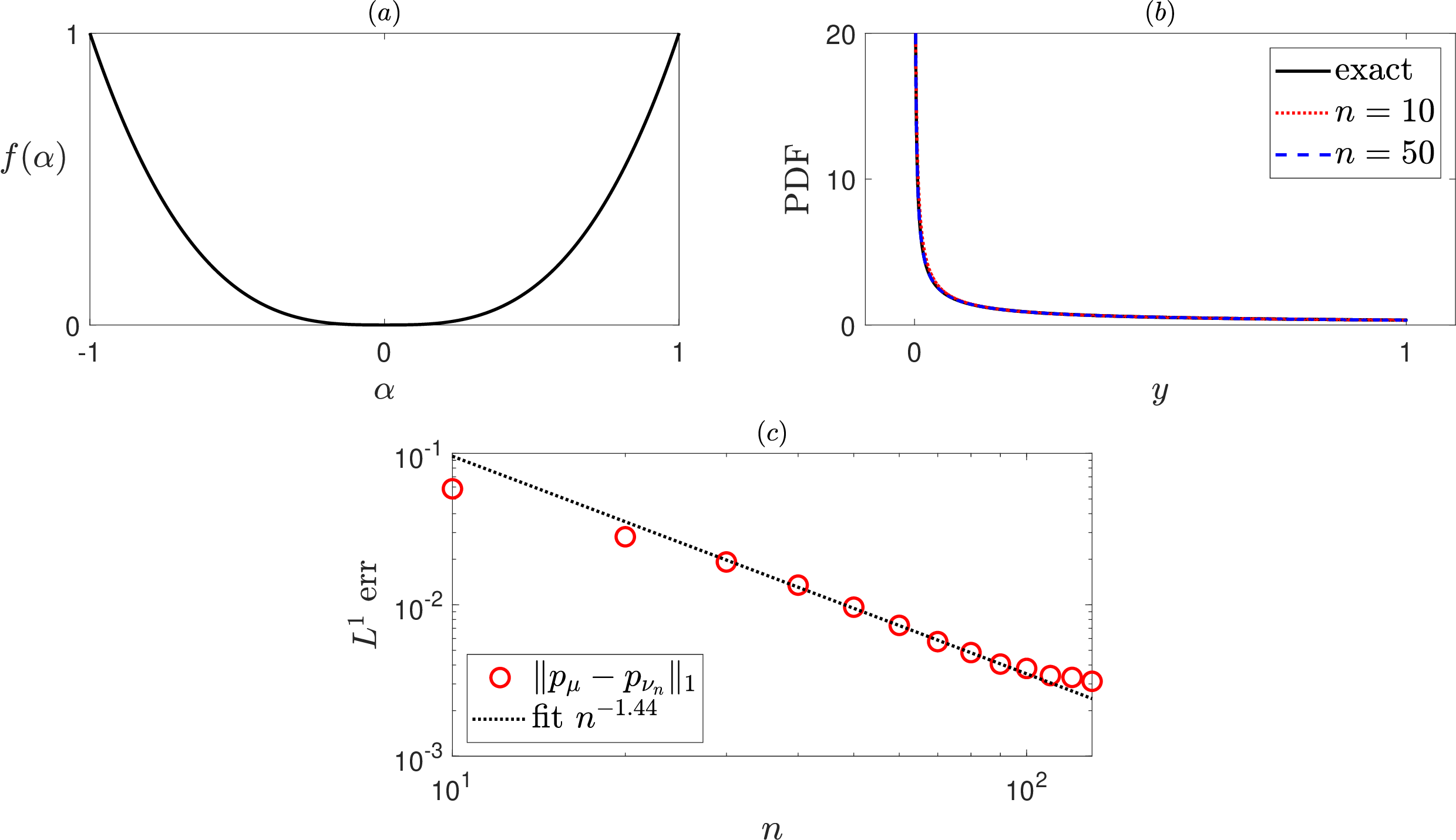}}
\caption{Same as Fig.\ \ref{fig:sin} for $f(\alpha) = |\alpha|^3$. In (c), a polynomial fit of $2.62n^{-1.44}$ is presented (dots, black).}
\label{fig:abscube}
\end{figure}

Our next two examples are of non-smooth functions. It is generally not advisable to approximate such functions with global polynomial methods such as gPC, and we do not promote this as a strategy in this work. Rather, we consider non-smooth functions to examine the sharpness and scope of our theory. In Fig.\ \ref{fig:abscube} we consider $f(\alpha)=|\alpha|^3$. This is case not covered by Theorem~\ref{thm:main_1d}, as $f\in~ H^3(\Omega)$ but not in $H^4(\Omega)$. Notwithstanding, we observe numerically that $\|p_{\mu}-~p_{\nu_n}\|_{L^1(\mathbb{R})} \lesssim n^{-1.44}$, which is comparable to $\|f-g_n\|_{H^1(\Omega)}$ which by \eqref{eq:canuto} converges as $n^{-1.5}$. Recall that the requirement in Theorem \ref{thm:main_1d} that $f\in H^6$ stems from our use of Sobolev embedding in Lemma \ref{lem:c2} to guarantee that $\|g_n\|_{C^2}$ is uniformly bounded in $n$. This boundedness is numerically satisfied in this example (results not shown), even though it is not guaranteed by our current theory.

Numerically, it seems that the $k$-dependence of the upper bound \eqref{eq:conv_sing} might not be tight. Since $f'(0)=f''(0)=0$, by applying \eqref{eq:conv_sing} with $k=2$ we expect that $\|p_{\mu}-p_{\nu_n}\|_{L^1(\mathbb{R})} \lesssim \|f-g_n\|_{H^1(\Omega)}^{1/5} \lesssim -n^{-0.3}$, which is much slower then what we observe in practice. Finally, consider a third function, $f(\alpha) = |\alpha - 0.5|$, which is in $H^1(\Omega)$ but not in $H^2(\Omega)$, see Fig.\ \ref{fig:abs}. While it is certainly not good practice to approximate such non-smooth functions with global polynomials, it is interesting to see that even so, $\|p_{\mu}-p_{\nu _n}\|_{L^1(\mathbb{R})} \lesssim n^{-0.78}$. In comparison, since $f\not\in H^2(\Omega)$, the theoretically-predicted $H^1(\Omega)$ convergence rate of $g_n$ to $f$ by \eqref{eq:canuto} is slower than $-1/2$, but was computed to be roughly $-0.6$ (results not shown). That convergence is obtained implies that stronger mechanisms of PDF convergence at at play than what our current theory accounts for.

\begin{figure}[h]
\centering
{\includegraphics[scale=.3]{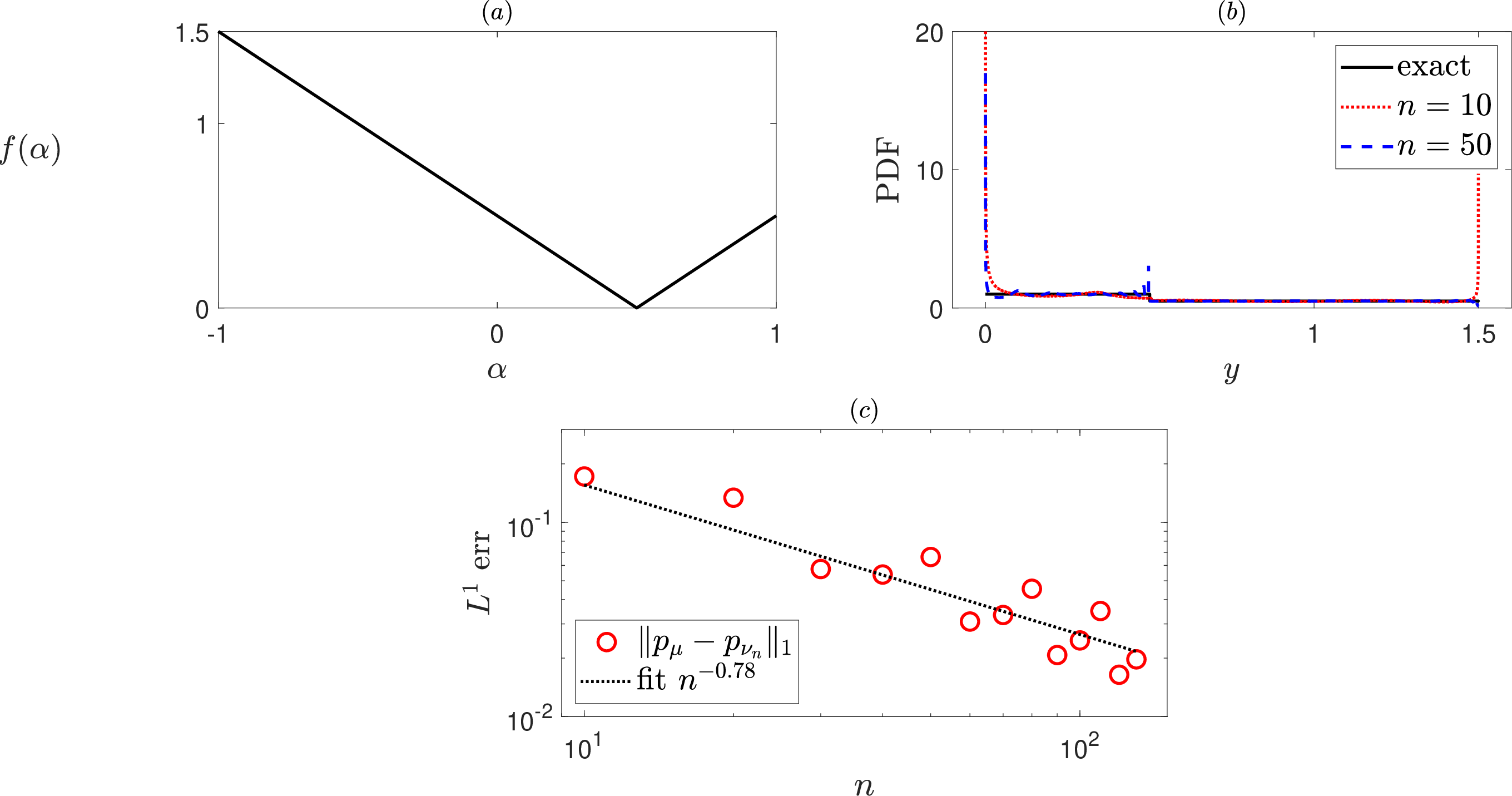}}
\caption{Same as Fig.\ \ref{fig:sin} for $f(\alpha) = |\alpha-0.5|$. In (c), a fit of $0.91n^{-0.78}$ is presented (dots, black).}
\label{fig:abs}
\end{figure}


This work also lays the foundations to the study of pushforwards by sparse-grid representations, which are of great practical importance. A key impediment to approximation in the multidimensional case $d>1$ is the so-called curse of dimensionality - the computational cost of constructing $g_n$ grows exponentially with the dimension $d$. For example, the collocation gPC on a tensor-grid of quadrature-points requires sampling $f$ at $n^d$ points, which becomes non-feasible already for relatively moderate dimensions. A common approach to this problem in $L^2$ approximation (which in the UQ context is equivalent to the approximation of moments) is by using sparse grid, which allow for superior orders of convergence \cite{back2011sparse, bungartz2004sparse, xiu2010numerical, xiu2005collocation}. Applying our current analytic approach to sparse grids would require a detailed analysis of the convergence of sparse quadratures in high-order Sobolev spaces.\footnote{Not to be confused with the more standard topic of $L^2(\Omega)$ or $H^1(\Omega)$ convergence sparse-representations of high-regularity functions $f\in H^s(\Omega)$ for $s>1$.} To the best of our knowledge, there are few relevant results in this direction, see~\cite{griebel2009optimized} and the references therein. However, as Theorem \ref{thm:main_1d} suggests for $d=1$, it might be that a sharper result is possible in $d>1$ as well. Such improvements of Theorem~\ref{thm:main_md} might relax the dependence in high-order Sobolev norms to $H^1(\Omega)$ convergence, a widely studied and well understood topic \cite{berthelmann1999high, bungartz2004sparse}.

\section{Acknowledgments}
The author would like to thank A.\ Ditkowski, G.\ Fibich, S.\ Steinerberger, L.~Tamellini, H.\ Wang, and M.I.\ Weinstein for useful comments and discussions. The author acknowledges the support of the AMS-Simons Travel Grant.

\bibliographystyle{alpha}

\begin{thebibliography}{1}
\bibitem{ablowitz2015interacting}
{\sc M.J. Ablowitz and T.P. Horikis.}
\newblock {\em Interacting nonlinear wave envelopes and rogue wave formation in deep water},
\newblock Phys. Fluids, 27 (2015), pp.~012107.

\bibitem{adams2003sobolev}
RA Adams and JJ Fournier,
\newblock {\em Sobolev Spaces,}
\newblock Elsevier, 2003.

\bibitem{back2011sparse}
J B{\"a}ck, F Nobile, L Tamellini, and R Tempone.
\newblock Stochastic spectral Galerkin and collocation methods for PDEs with random coefficients: a numerical comparison. 
\newblock In J.S. Hesthaven and E.M. Ronquist, editors, {\em Spectral and High Order Methods for Partial Differential Equations,} volume 76 of Lecture Notes in Computational Science and Engineering, pages 43–62. Springer, 2011.

\bibitem{berthelmann1999high}
V Berthelmann, E Novak, and K Ritter, 
\newblock High dimensional polynomial interpolation on sparse grids, 
\newblock {\em Adv.\ Comput.\ Math.}, 12, 1999, pp.\ 273--288.


\bibitem{bungartz2004sparse}
HJ Bungartz and M Griebel,
\newblock Sparse grids,
\newblock {\em Acta Numerica,} 13, 2004, pp.\ 147--269.

\bibitem{butler2018convergence}
T Butler, J Jakeman, and T Widely,
\newblock Convergence of probability densities using approximate models for forward and inverse problems in uncertainty quantification,
\newblock {\em SIAM J. Sci. Comput.}, 40, 2018, A3523–A3548.

\bibitem{chae2020wass}
Chae M and Walker SG, 
\newblock Wasserstein upper bounds of the total variation for smooth densities,
\newblock {\em Stats. Probab. Lett.,} 163 (2020), pp. 108771.


\bibitem{canuto1982approximation}
C Canuto and A Quanteroni,
\newblock{Approximation results for orthogonal polynomials
in Sobolev spaces,}
\newblock{\em Math. Comput.}, 38 (1982), pp. 67--86.



\bibitem{capodaglio2018approximation}
G Capodaglio, M Gunzburger, and HP Wynn,
\newblock Approximation of probability density functions for SPDEs using truncated series expansions,
\newblock {\em arXiv preprint} arXiv:1810.01028, 2018.

\bibitem{chen2005uncertainty}
{\sc Q.Y Chen, D.~Gottlieb, and J.S. Hesthaven.}
\newblock {\em Uncertainty analysis for the steady-state flows in a dual throat nozzle,} 
\newblock J. Comput. Phys., 204 (2005), pp.~378--398.


\bibitem{colombo2018basins}
{\sc I. Colombo, F. Nobile, G. Porta, A. Scotti, and L. Tamellini.}
\newblock {\em Uncertainty quantification of geochemical and mechanical compaction in layered sedimentary basins,}
\newblock Comput. Methods Appl. Mech. Engrg., 328 (2018), pp.~122--146.


\bibitem{constantine2012sparse}
PG Constantine, MS Eldred, and ET Phipps,
\newblock Sparse pseudospectral approximation method,
\newblock{\em Comput. Methods Appl. Mech. Engrg.,} 229
 (2012), pp. 1--12.
 
\bibitem{davis1975interpolation}
PJ Davis,
\newblock {\em Interpolation and Approximation,}
\newblock Dover, 1975.

\bibitem{davis1967integration}
PJ Davis and P Rabinowitz,
\newblock {\em Numerical {I}ntegration,}
\newblock Academic, New York, 1975.


\bibitem{sagiv2020uq}
A Ditkowski, G Fibich, and A Sagiv,
\newblock{Density estimation in uncertainty propagation problems using a surrogate
model,}
\newblock{\em SIAM/ASA J. Uncertain. Quantif.}, 8 (2020), pp. 261-300.

\bibitem{ditkowski2019spectral}
A Ditkowski and R Katz,
\newblock On spectral approximations with nonstandard weight functions and their implementations to generalized chaos expansions,
\newblock {\em J.\ Sci.\ Comput.,} 79 (2019), pp.\ 1985--2005.


\bibitem{estep2009nonparametric}
D Estep, A Malqvist, and S Tavener,
\newblock Nonparametric density estimation for randomly perturbed elliptic problems I: computational methods, a posteriori analysis, and adaptive error control,
\newblock {\em SIAM J. Sci. Comput.}, 31, 2009, pp.\ 2935--2959.

\bibitem{evans_pde}
LC Evans,
\newblock {\em Partial Differential Equations,}
\newblock (Vol. 19) American Mathematical Society, 2010.

\bibitem{zabaras2007sparse}
{\sc B. Ganapathysubramanian and N. Zabaras.}
\newblock {\em Sparse grid collocation schemes for stochastic natural convection problems,}
\newblock J. Comp. Phys., 225 (2007), pp.~652--685.

\bibitem{ghanem2017handbook}
{\sc R. Ghanem, D. Higdon, and H. Owhadi.}
\newblock {\em Handbook of {U}ncertainty {Q}uantification,}
\newblock Springer, New York, 2017.

\bibitem{ghanem2003stochastic}
{\sc R. Ghanem and P.D. Spanos.}
\newblock {\em Stochastic Finite Elements: a Spectral Approach,}
\newblock Springer-Verlag, New-York, 1991.

\bibitem{gibbs2002choosing}
A.L.\ Gibbs and F.E.\ Su,
\newblock {\em On choosing and bounding probability metrics,}
\newblock Int. Stats. Rev., 70:419--435, 2002.

\bibitem{griebel2009optimized}
M Griebel and S Knapek,
\newblock Optimized general sparse grid approximation spaces for operator equations,
\newblock {\em Math. Comput.,} 78, 2009, pp.\ 2223--2257.

\bibitem{gotlieb2007spectral}
{\sc J.S. Hesthaven, S. Gottlieb, and D. Gottlieb.}
\newblock {\em Spectral Methods for Time-Dependent Problem,}
\newblock Cambridge, UK, 2007.

\bibitem{lemaitre2004wavelet}
{\sc O.P. Le~{M}a{\^\i}tre, O.M. Knio, H.N. Najm, and R. Ghanem.}
\newblock {\em Uncertainty propagation using Wiener--Haar expansions,}
\newblock J. Comp. Phys., 197 (2004), pp.~28--57.

\bibitem{le2012asymptotics}
{\sc L. Le~Cam and G.L. Yang.}
\newblock {\em Asymptotics in {S}tatistics: {S}ome {B}asic {C}oncepts,}
\newblock Springer Science \& Business Media, New York, 2012.

\bibitem{le2010asynchronous}
{\sc O.P. Le~Ma{\^\i}tre, L. Mathelin, O.M. Knio, and M.Y. Hussaini.}
\newblock {\em Asynchronous time integration for polynomial chaos expansion of uncertain periodic dynamics,}
\newblock {\em Discrete Contin. Dyn. Syst}, 28 (2010), pp.~199--226.



\bibitem{o2013polynomial}
{\sc A. O'Hagan.}
\newblock {\em Polynomial chaos: A tutorial and critique from a statistician’s perspective.}
\newblock SIAM/ASA J. Uncertain. Quantif., 20 (2013), pp.~1--20.


\bibitem{patwardhan2017loss}
{\sc G. Patwardhan, X. Gao, A. Sagiv, A. Dutt, J. Ginsberg, A. Ditkowski, G. Fibich, and A.L. Gaeta.}
\newblock {\em Loss of polarization of elliptically polarized collapsing beams.}
\newblock Phys. Rev. A, 99 (2019), pp.~033824.


\bibitem{piazzola2020uq}
Piazzola C, Tamellini L, Pellegrini R, Broglia R, Serani A, and Diez M,
\newblock Uncertainty quantification of ship resistance via multi-index stochastic collocation and radial basis function surrogates: a comparison,
\newblock {\em AIAA Aviation 2020 Forum,} pp.~3160, 2020.

\bibitem{quarteroni2000numerical}
A Quarteroni, R Sacco, and F Salero,
\newblock {\em Numerical Mathematics,}
\newblock Springer-Verlag, New York NY, USA, 2000.



\bibitem{sagiv2020wasserstein}
A Sagiv,
\newblock The Wasserstein distances between pushed-forward measures with applications to uncertainty quantification,
\newblock {\em arXiv preprints,} arXiv:1902.05451 (to appear in {\em Comm. Math. Sci.})


\bibitem{best2017paper}
{\sc A Sagiv, A Ditkowski, and G Fibich,}
\newblock {\em Loss of phase and universality of stochastic interactions between laser beams,}
\newblock Opt. Exp., 25 (2017), pp.~24387--24399.

\bibitem{salvemini1943sul}
T.\ Salvemini,
\newblock {\em Sul calcolo degli indici di concordanza tra due caratteri quantitativi,} \newblock Atti della I Riunione della Soc. Ital. di Statistica, Roma, 1943.

\bibitem{santa2015optimal} F Santambrogio,
\newblock {\em Optimal Transport for Applied Mathematicians. Calculus of Variations, PDEs, and Modeling,
Progress in Nonlinear Differential Equations and their Applications,}
\newblock Birk{\"a}user, New York, 2015.

\bibitem{stefanou2009stochastic}
{\sc G. Stefanou.}
\newblock {\em The stochastic finite element method: past, present and future,}
\newblock Comput. Methods Appl. Mech. Engrg., 198 (2009), pp.~1031--1051.



\bibitem{sudret2000stochastic}
{\sc B. Sudret and A. Der~Kiureghian.}
\newblock {\em Stochastic {F}inite {E}lement {M}ethods and {R}eliability: a {S}tate-of-the-{A}rt {R}eport,}
\newblock Department of Civil and Environmental Engineering, University of California Berkeley, Berkeley, CA, 2000.

\bibitem{szego1939orthogonal}
{\sc G. Szego.}
\newblock {\em Orthogonal {P}olynomials}, Amer. Math. Soc. Colloq. Publ., 23,
\newblock American Mathematical Society, New York, 1939.


\bibitem{tipireddy2014basis}
R Tipireddy and R Ghanem,
\newblock Basis adaptation in homogeneous chaos spaces,
\newblock {\em J. Comput. Phys.,} 259, 2014, pp.\ 304--317.



\bibitem{trefethen2013book}
{\sc L.N. Trefethen.}
\newblock {\em Approximation Theory and Approximation Practice,}
\newblock SIAM, Philadelphia, PA, 2013.


\bibitem{tsybakov2009estimate}
{\sc A.B. Tsybakov.}
\newblock{\em Introduction to Nonparametric Estimation,}
\newblock Springer, New York, 2009.


\bibitem{ullmann2014pod}
{\sc S. Ullmann and J. Lang.}
\newblock {\em POD-Galerkin modeling and sparse-grid collocation for a natural convection problem with stochastic boundary conditions,}
\newblock in Sparse Grids and Applications, Munich, Spring (2012), pp.~295--315.


\bibitem{vallender1974calculation}
S.S.\ Vallender,
\newblock {\em Calculation of the {W}assertein distance between probability distributions on the line,}
\newblock SIAM Theory Prob. Appl., 18:784--786, 1974.

\bibitem{villani2003topics}
C Villani,
\newblock {\em Topics in {O}ptimal {T}ransportation,} American Mathematical Society, 2003.

\bibitem{wan2005adaptive}
X Wan and GE Karniadakis,
\newblock {\em An adaptive multi-element generalized polynomial chaos method for stochastic differential equations,}
\newblock J. Comput. Phys., 209 (2005), pp.~617--642.

\bibitem{wang2020fast}
H Wang,
\newblock How fast does the best polynomial approximation converge than Legendre projection?
\newblock {\em Numer. Math.}, 147 (2021), pp.~481--583.

\bibitem{wang2012convergence}
H. Wang and S. Xiang,
\newblock {\em On the convergence rates of {L}egendre approximation,}
\newblock Math. {C}omp., 81 (2012), pp.~861--877.

\bibitem{wasserman2006allnon}
  L Wasserman, 
  \newblock{ \em All of Nonparametric Statistics,}
  \newblock Springer Science \& Business Media, 2006.


\bibitem{wasserman2013all}
L Wasserman,
\newblock {\em All of Statistics: A Concise Course in Statistical Inference,}
\newblock Springer Science \& Business Media, New York, 2004.

\bibitem{xiu2010numerical}
{\sc D. Xiu.}
\newblock {\em Numerical {M}ethods for {S}tochastic {C}omputations: a
  {S}pectral {M}ethod {A}pproach}.
\newblock Princeton University, Princeton, NJ, 2010.

\bibitem{xiu2005collocation}
{\sc D. Xiu and J.S. Hesthaven.}
\newblock {\em High-order collocation methods for differential equations with random inputs.}
\newblock SIAM J. Sci. Comput., 27 (2005), pp.~1118--1139.


\bibitem{xiu2002galerkin}
{\sc D. Xiu and G.E. Karniadakis.}
\newblock {\em The {W}iener--{A}skey polynomial chaos for stochastic differential equations,}
\newblock SIAM J. Sci. Comput., 24 (2002), pp.~619--644.

\bibitem{zech2020sparse}
J Zech and Y. Marzouk,
\newblock Sparse approximation of triangular transports on bounded domains,
\newblock{ \em arXiv preprint} arXiv:2006.06994, 2020.

\end{thebibliography}

\end{document}